\documentclass[12pt,a4paper]{amsart}%

\usepackage{imakeidx}
\usepackage{amsmath}
\usepackage{fullpage}
\usepackage{comment}
\usepackage{MyCommA}
\usepackage[final]{pdfpages}
\usepackage{tikz-cd}

\makeindex

\newcommand{\RamiA}[1]{{{#1}}}
\newcommand{\RamiB}[1]{{{#1}}}
\newcommand{\Eitan}[1]{{{#1}}}
\newcommand{\Dima}[1]{{{#1}}}
\newcommand{\Nir}[1]{{{#1}}}

\newcommand{\NextVer}[1]{}

\newcommand{\noleft}{\left.\kern-\nulldelimiterspace}

\DeclareMathOperator{\trace}{tr}

\DeclareMathOperator{\Gal}{Gal}
\newcommand{\hgt}{\mathfrak{H}}

\title{Push forward\Dima{s} of smooth measures on the unitary group in odd characteristic under the Chevalley map are almost bounded}

\author[Aizenbud]{Avraham Aizenbud}
\address{Avraham Aizenbud,
Faculty of Mathematical Sciences,
Weizmann Institute of Science,
76100 Rehovot, Israel}
\email{aizenr@gmail.com}
\urladdr{https://www.wisdom.weizmann.ac.il/~aizenr/}

\author[Avni]{Nir Avni}
\address{Nir Avni,
Department of Mathematics, 
Northwestern University,
2033 N. Sheridan Rd,
Evanston IL, 60201}
\email{avni.nir@gmail.com}

\author[Gourevitch]{Dmitry Gourevitch}
\address{Dmitry Gourevitch,
Faculty of Mathematical Sciences,
Weizmann Institute of Science,
76100 Rehovot, Israel}
\email{dimagur@weizmann.ac.il}
\urladdr{https://www.wisdom.weizmann.ac.il/~dimagur/}

\author[Kazhdan]{David Kazhdan}
\address{David Kazhdan,
Einstein Institute of Mathematics,
Edmond J. Safra Campus, Givaat Ram,
The Hebrew University of Jerusalem,
Jerusalem 91904, Israel}
\email{david.kazhdan@mail.huji.ac.il}
\urladdr{https://math.huji.ac.il/~kazhdan/}

\author[Sayag]{Eitan Sayag}
\address{Eitan Sayag,
Department of Mathematics,
Ben Gurion University of the Negev,
P.O.B. 653,
Be'er Sheva 84105, Israel}
\email{eitan.sayag@gmail.com}
\urladdr{https://www.math.bgu.ac.il/~sayage}
\date{\today}

\begin{document}

\maketitle
\begin{abstract}
    In this paper we study analytic properties of the Chevalley map 
    $p:\fu \to \fc$ where $\fu$ is the Lie algebra of a unitary group over a local field of odd characteristics. This map essentially sends a matrix to its characteristic polynomial. 
 Our main result is that smooth compactly supported measures on $\fu$ are pushed by the Chevalley map to measures on $\fc$ with density functions in $L^q(\fc)$ for all $q < \infty$. We also provide similar results for the Chevalley map attached to the unitary group.
    
    An important consequence is an odd characteristic analogue to results of Harish-Chandra stating the local integrability of the orbital integrals on the \Dima{unitary group and its Lie algebra.}
    Further possible applications to local integrability of characters of admissible representations of unitary groups over local fields of odd characteristic are also discussed. 
\end{abstract}

\tableofcontents

\section{Introduction}
\subsection{Main results}
    Let $\bfU$ be a unitary group defined over a local field $F:=\F_\ell((t))$ of odd characteristic and let $\uu$ be its Lie algebra. Let $p:\uu\to \uc:=\uu//\bfU$ be the Chevalley map (see more details in \Cref{lem:qust.lie} below). Let $\fu:=\uu(F)$ and $\fc:=\uc(F)$. Let $\mu_{\fu}$ and $\mu_{\fc}$ be Haar measures on $\fu$ and $\fc$.
    In this paper we prove:
\begin{introtheorem}[\S\ref{sec:key.ingr}]\label{thm:main}
    For every $f\in C_c^\infty(\fu),$    
    we have $$\frac{p_*(f \mu_{\fu})}{\mu_{\fc}}\in \bigcap_{q\in (1,\infty)} L^q({\fc}).$$ 
\end{introtheorem}
We also deduce the following group version of this theorem.
Let $\cP:\bfU\to \bfC:=\bfU//\bfU$ be the group Chevalley map (see \S\ref{ssec:cat.qu} for more details). Let $U:=\bfU(F)$ and $C:=\bfC(F)$.
\begin{introtheorem}[\S\ref{sec:grp.case}]\label{thm:main.grp}
Let  $\mu_U$ be a Haar measure on $U$ and $\mu_C$ be a smooth nowhere vanishing measure on $C$ (see \Cref{def:FRS} for more details). 
    For every $f\in C_c^\infty(U),$    
    we have $$\frac{\cP_*(f \mu_{U})}{\mu_{C}}\in \bigcap_{q\in (1,\infty)} L_{loc}^q(C).$$
\end{introtheorem}
\subsubsection{Relation with orbital integrals}
In view of the results of \cite{DKS,Bes} one can deduce from Theorems \ref{thm:main}  and \ref{thm:main.grp} the local integrability of orbital integrals. Let us briefly remind what they are.
\Nir{
\begin{notation}
Fix a Haar measure $\mu_U$ on $U$.
\begin{enumerate} 
\item For a regular semisimple element $x$ in either $U$ or $\fu$, let $U_x \subseteq U$ be the stabilizer of $x$ and let $\mu_x$ be the Haar measure on $U_x$ normalized so that its maximal compact has measure $1$.
\item Let $f$ be in $C_c^\infty(U)$  or $C_c^\infty(\fu)$. If $x$ is a regular semisimple element as above, define $$\mdef{\Omega(f)}(x)=\int_{U/U_x} f(g x g^{-1})dg$$ where $dg$ is the measure induced by $\mu_U$ and $\mu_x$.
\end{enumerate}
$\Omega(f)$ is an almost everywhere defined function on $U$ or $\fu$.
\end{notation}
}

\begin{introcorollary}\label{thm:orb}
$ $
\begin{enumerate}
    \item     Let $f\in C_c^\infty(U)$. Then $\Omega(f)\in L^{1+\eps}(U)$ for some $\eps>0$.
\item 
    Let $f\in C_c^\infty(\fu)$. Then $\Omega(f)\in L^{1+\eps}(\fu)$ for some $\eps>0$.
\end{enumerate}
\end{introcorollary}
\RamiB{
\subsubsection{Status of the results}
The results of this paper depend on the drafts \cite{Bes,DKS}. We expect these drafts to appear on the arXiv in the next few months.
}
\subsection{Background}
\subsubsection{Related results}
A stronger form of the characteristic $0$ counterpart of \Cref{thm:orb} for arbitrary connected reductive groups was proved by Harish-Chandra (cf. \cite[Theorem 13 and Part VII]{HC_VD}) by different methods. The variant of  \Cref{thm:main} with $\bfU$ replaced by $\GL_n$  for arbitrary \Nir{characteristic} was proved in \cite{AAGKS}. We refer the reader to \cite[\S 1.3]{AGKS_jets}, \cite[\S 1.2]{AAGKS} for further discussions of the background of this problem.
\subsubsection{Motivation}
For us, the main motivation for \Cref{thm:main.grp} is establishing an analog of Harish-Chandra's integrability theorem \cite{HC_int} in positive characteristic. Harish-Chandra's integrability theorem states that the 
character of every irreducible representation of a p-adic group is 
represented
by a locally $L^1$ function. Our aim is to extend this theorem for groups defined over local fields of positive characteristic. In \cite{AGKS_HC,AAGKS} we did this for cuspidal representations of $\GL_n(F)$. 

The proofs in \cite{HC_int,AGKS_HC,AAGKS} depend on (some versions of) the relevant counterparts of Theorems \ref{thm:main}, \ref{thm:main.grp}, \ref{thm:orb}. So we consider the current work as a step towards Harsh-Chandra's integrability theorem for unitary groups in odd characteristic.

As mentioned in \cite{Bes}, in order to deduce the Harsh-Chandra's integrability for cuspidal representations of the group $U$ from \Cref{thm:orb} it would be sufficient to establish an analog of \cite{AGKS_Orb} for the group $\bfU$. However, rather than doing this we plan (in the works in progress \cite{AGKS_eff,AGKS_int}) to deduce local integrability of characters of arbitrary  representations of $U$ from \Cref{thm:orb}.


\subsection{Ideas of the proofs}
\subsubsection{\Cref{thm:main}}
{For a unitary Lie algebra $\fu$ as above we introduce a measurable function $\kappa_{\fu}:\fu \to \C$ (see \S \ref{ssec:kap}).}
The proof of \Cref{thm:main}
is by induction. It 
consists of 2 main steps:
\begin{enumerate}
    \item If \Cref{thm:main} holds, then 
    $\kappa_{\fu}$ is locally integrable. See \Cref{thm:chv.kappa} below.
    \item If \Nir{the function} 
    $\kappa_{\fu'}:\fu' \to \C$, for a certain unitary Lie algebra $\fu'$ of smaller rank, is locally integrable then \Cref{thm:main} holds (for $\fu$). See \Cref{thm:kappa.chv} below.
\end{enumerate}

Step (1) is done in \cite{Bes}. The main part of the proof is Step (2). For this, we decompose $\fu$ to a direct sum of 3 subspaces $\fu'\oplus W\oplus F$ ($\fu'$ is the unitary Lie algebra of a co-dimension 1 subspace) such that the map $p$ is quadratic on $W$. \Nir{We show that, for a generic element $A\in \fu'$, the restriction of $p$ to $\left\{ A \right\} \oplus W\oplus F\to \fc$ }pushes a smooth compactly supported measure to a measure with $L^q$-density function for any $q\in [1,\infty)$. Moreover, we bound the $L^q$-norm of this density in terms of {$\kappa_{\fu'}(A)$} \Nir{and the norm of the restriction}.

\subsubsection{Comparison with the case of $\GL_n$}
The proof is somewhat analogous to the proof in \cite{AAGKS} of the case of $\fg\fl_n$. In 
\cite{AAGKS} we also have a similar decomposition of $\fg\fl_n$. However, while here the map $p$ on $W$ is quadratic (and $\dim W$ is approximately twice the dimension of $\fc$), in the $\fg\fl_n$-case the Chevalley map on the analogous space is linear (and it is of the same dimension as $\fc$). This leads to a simpler estimate on the relevant $L^q$ norm, and thus we can show directly that this estimate is locally integrable and we do not need the inductive procedure.

So, while in the $\fg\fl_n$-case, the analog of \Cref{thm:main} is a step preceding the proof of the local integrability of $\kappa$ (which then can be used for the local integrability of orbital integrals and then for characters of cuspidal representations), here \Cref{thm:main} is proven by induction together with the local integrability of $\kappa$.
\subsubsection{\Cref{thm:main.grp}}
\Cref{thm:main.grp} is deduced from \Cref{thm:main} by a standard technique based on the Cayley transform. See \S \ref{Subsec-proof-of-thmB}. 

\subsection{Applicability for other groups}
We expect the argument in the current paper to be applicable to orthogonal groups. However, it is not applicable (at least in the current form) to the symplectic group. 

Therefore we plan to proceed using  a different strategy: we would like to use the original Harish-Chandra's argument to deduce from the unitary case (together with the $\fg\fl_n$ case) the counterpart of \Cref{thm:main} for many other groups including all classical groups in odd characteristic. 

\Eitan{Let us explain why we expect this strategy to work.}  
The main point where the classical Harish-Chandra's argument \Nir{breaks in positive characteristic is that the addition map
\[
\cI:Z(\ug) \times \mathcal{N}_\ug \rightarrow \ug
\]
from the product of the center and the nilpotent cone of the Lie algebra to the Lie algebra is not a closed immersion.

This phenomenon is rare  outside type A. We expect Harish-Chandra's argument to show that if $\cI$ is an immersion and \Cref{thm:main} is valid for centralizer of non-central semi-simple elements in $\ug$ then \Cref{thm:main} holds for $\ug$ (at least in good characteristic).

This should solve the problem for  the symplectic group  since in this case, these centralizers are products of smaller symplectic groups and unitary group. 

We expect such an argument to work for \Dima{a} large family of groups.
 We will work on it in \cite{AAGKS_cl}.
}

\subsection{Structure of the paper}
In \S\ref{sec:settings} we fix our notation, recall the basic facts on the unitary group and on its Chevalley map, and define the function $\kappa$.

In \S\ref{sec:key.ingr} we formulate the two ingredients of the proof of \Cref{thm:main}, namely \Cref{thm:chv.kappa} and \Cref{thm:kappa.chv}, and deduce \Cref{thm:main} from them by induction on the dimension of $V$.

In \S\ref{sec:kottwitz} we recall the Kottwitz height of a quasi-affine variety and record the properties of it that we need, in particular its local integrability.

In \S\ref{sec:Lr.bound} we introduce, for regular semi-simple $A\in \fu$, a map $\phi_A$ and prove \Cref{thm:Lr-bound}: the map $\phi_A$ pushes a smooth compactly supported measure to a measure with $L^{<\infty}$ density, and the $L^r$-norm of this density is bounded by $\kappa(A)$ times a locally integrable function. This is the main ingredient in the proof of \Cref{thm:kappa.chv}.

In \S\ref{sec:pf.kappa.chv} we prove \Cref{thm:kappa.chv}. We decompose $\fu \cong \fu_W\oplus W \oplus F$, describe the Chevalley map in terms of the maps $\phi_A$ (\Cref{p and phi}), and deduce the statement from \Cref{thm:Lr-bound} using the Minkowski integral inequality.

Finally, in \S\ref{sec:grp.case} we deduce \Cref{thm:main.grp} from \Cref{thm:main} using the Cayley transform.
\subsection{Acknowledgments}
During the preparation of this paper, A.A., D.G. and E.S. were partially supported by the ISF grant no. 1781/23, A.A. and N.A. were partially supported by BSF grant no. 2022193, N.A. was partially supported by NSF grant no. 2503233, and D.K. was partially supported by an ERC grant 101142781.
\section{Settings}\label{sec:settings}
\subsection{Conventions}\label{Sec.Conv.}

Throughout the paper:

\begin{itemize}

\item \tdef{$\ell$} is an odd prime power, $\mdef{F}:=\F_\ell((t))$, $\mdef{O_F}:=\F_\ell[[t]]$.
\item Unless stated otherwise, the \tdef{algebraic varieties} that we consider are defined over $F$.
{
\item We will use two conventions to name an algebraic variety and its $F$-rational points. In one, an algebraic variety is denoted by a boldface letter (e.g. $\bfX$) and the set of $F$ rational points is denoted by the corresponding unbolded letter (e.g. $X$). In the other, an algebraic variety is denoted by an underlined gothic letter (e.g. $\uu$) and the set of $F$ rational points is denoted by removing the underline (e.g. $\fu$).
}
\item We use the same notation for algebraic maps and the induced maps on $F$-points.

\item For an algebraic variety $\mathbf X$ over $F$, we consider the collection of its $F$-points as an $l$-space (in the sense of \cite{BZ}). If the variety is smooth we also consider this collection as an $F$-analytic  manifold.
\item For an $l$-space $X$,
we denote by \tdef{$C_c^\infty(X)$} the space of locally constant compactly supported functions on $X$ and by \tdef{$C^\infty(X)$} the space of locally constant functions on $X$.
\item For any finite dimensional $F$-vector $W$ we choose a Haar measure \tdef{$\mu_W$} on it.
\item For a {continuous function $\phi:X\to Y$ of topological spaces} and a compactly supported {(Radon)} measure $\mu$ on $X$, we denote by \tdef{$\phi_*(\mu)$} the pushforward measure on $Y$.

\item For $q\in (0,\infty]$,
by an \tdef{$L_{{loc}}^q$}, or \tdef{$L^q$} function on an $F$-vector space we mean a function which is $L_{{loc}}^q$ or $L^q$ with respect to $\mu_W$. We will use these notions also for functions that are only defined almost everywhere.
\item Also denote $$\mdef{L^{<q}}=\bigcap_{r< q}L^{r}; \quad \mdef{\mdef{L_{loc}^{<q}}}=\bigcap_{r< q}L_{loc}^{r}; \quad  \mdef{L^{>q}}=\bigcup_{r> q}L^{r}; \quad   \mdef{L_{loc}^{>q}}=\bigcup_{r> q}L_{loc}^{r}$$

\end{itemize}

\subsection{Unitary group}\label{ssec:prel.uni}
Fix an etale algebra  \tdef{$E/F$} of degree 2. $E$ is either a quadratic field extension or the algebra $F\times F$. We write $x\mapsto \mdef{\overline{x}}$ for the nontrivial Galois involution of $E/F$. Choose a {non-zero} element
$
\mdef{\xi}\in E,
$
such that
$
\overline{\xi}=-\xi.
$

Let $\mdef{\bfV}$ be restriction of scalars of $\bA^n$ from $E$ to $F$. In particular
 $\mdef{V}:=\bfV(F)=E^n$.  Let $$
\mdef{B}:V\times V\longrightarrow E
$$
be a non-degenerate Hermitian form on $V$ with respect to the extension  $E/F$.
The unitary group of $V$ is the $F$-algebraic group:
$$
\mdef{\bfU}:=\mdef{\bfU_V}
:=
\left\{
g\in \GL_E(\bfV)|
B(gv,gw)= B(v,w)
\text{ for all }v,w\in V
\right\}.
$$
Its 
Lie algebra is 
$$
\mdef{\uu}:=\mdef{\uu_V}
:=
\left\{
X\in \End_E(\bfV)|
B(Xv,w)+B(v,Xw)=0
\text{ for all }v,w\in V
\right\}.
$$

\Dima{
\begin{lemma} \label{unitary if F x F} Let $\overline{V}$ be the conjugate vector space of $V$, i.e., the additive group $V$ where the multiplication of $v$ by $\alpha \in E$ is $\overline{\alpha}v$. 
\begin{enumerate}
    \item \label{it:Biso} The non-degenerate Hermitian form $B$ gives rise to an isomorphism between $V$ and the dual space of $\overline{V}$.
    \item \label{it:BSplit} If $E=F \times F$ then this isomorphism gives an isometry between $(V,B)$ and $( F^n \times (F^n)^* , B_s)$, where $B_s((v,v^*),(u,u^*))=(v^*(u),u^*(v))\in E$. Under this isometry, 
    $$\bfU=\left\{ (g,(g ^{-1})^T) \mid g\in \GL_n \right\} \text{ and }\fu= \left\{ (X,-X^T) \mid X\in \mathfrak{gl}_n \right\}.$$ 
\end{enumerate}
 \end{lemma} 
\begin{proof}
\eqref{it:Biso} is clear. For \eqref{it:BSplit}, note that {if $(g,h)\in \bfU_{B_s}$ then for every $v\in V$ and $u^*\in V^*$ we have
\begin{multline*}  
(0,(hu^*)(gv))=B_s((gv,0),(0,hu^*))=B_s((g,h)(v,0),(g,h)(0,u^*))\\=B_s((v,0),(0,u^*))=(0,u^*(v))
\end{multline*}
thus $
u^*(h^Tgv)=(hu^*)(gv)=u^*(v)$ so $h^Tg=I$.}
The result for $\fu$ follows immediately.
\end{proof}
}

Let
$$
\mdef{\uu^{rss}}:=\mdef{\uu(V)^{rss}}\subset \uu(V),
$$ denote the regular semisimple locus of $\uu$
and let $\mdef{\Delta}:=\mdef{\Delta_{V}}\in O_\uu(\uu)$ be the Weyl discriminant (See \cite[7.2]{Kot}).

\subsection{Invariants and the Chevalley map}\label{subsec: chev.map}
In the following we recall the ring of invariants for the adjoint action of the unitary group on its Lie algebra, fix generators for this ring and use this to explicate the Chevalley map  $p:\uu\to \uc:=\uu//\bfU$ in these coordinates. 

\begin{notation}
    For $A\in \fu$ let $\mdef{\chi_A}=x^n +\sum_{i=0}^{n-1}  \mdef[c_{n}(A)]{c_{n-i}(A)}x^i$ be the   characteristic polynomial of $A$. For $1 \leq i \leq n,$ let $\mdef{ b_i(A)}:=\frac{c_i(A)}{\xi^i}$. This defines polynomials $b_i\in O_\uu(\uu)^{\bfU}$ (with coefficients in $F$).
\end{notation}
The following lemma is standard. For completeness, we include its proof here.
\begin{lem}\label{lem:qust.lie}
    The algebra $O_\uu(\uu)^{\bfU}$ of invariants is  a polynomial algebra in the elements $(b_i)_{i=1}^{n}.$
\end{lem}
\begin{proof}
$ $
\begin{enumerate}[{Case} 1:]
    \item $E\cong F\times F$.\\
    In this case the statement becomes the analogous statement for $\GL_n$ which is well known, see e.g. \cite[Theorem 1.3]{DeCo-Proc}. 
    \item 
    Consider the map $$F[x_1,\dots,x_n]\to O_\uu(\uu)^{\bfU}$$ sending $x_i$ to $b_i$. We need to show that it is an isomorphism. For this it is enough to show that the corresponding map     $$F[x_1,\dots,x_n]\otimes_F E \to O_\uu(\uu)^{\bfU}\otimes_F E$$
    is an isomorphism. This follows from the previous step as $E\otimes_F E\cong E\times  E$. 
\end{enumerate}
   
\end{proof}

We move on to describe the Chevalley map for the Lie algebra of the unitary group. 
Denote $\mdef{\uc}:=\mdef{\uc_V}:=\uu//\bfU:=\Spec(O_\uu(\uu)^{\bfU})$ and identify it with $\bA^n$ via the polynomials $b_i$. Denote by $\mdef{p}:=\mdef{p_V}:\uu\to \uc$ the quotient map. This is the Chevalley map of the Lie algebra $\uu$.

Let $\mdef{\fu}:=\mdef{\fu_V}:=\uu(F)$,  $\mdef{\fu^{rss}}:=\mdef{\fu^{rss}_V}:=\uu^{rss}(F)$,  $\mdef{U}:=\mdef{U_V}:=\bfU(F)$ and $\mdef{\fc}:=\mdef{\fc_V}:=\uc(F)$. 

{
\subsection{The function $\kappa$}\label{ssec:kap}

{
\begin{lemma}[{\cite[\S4.1]{AGKS_HC}}]\label{lem:kaz.form}
 Let $\mathbf{T}$ be a $d$-dimensional torus defined over $F$. 
 
 \begin{enumerate}
     \item 

 There is a separable field extension $L/F$ and an isomorphism $$\tau:\mathbf{T} \times \Spec(L) \rightarrow \mathbb{G}_m^d \times \Spec(L).$$
 
 \item 
 The top form $\omega_\tau=\tau ^* \left( \frac{d t_1}{t_1} \wedge \cdots \wedge \frac{d t_d}{t_d} \right)\in \Omega^d \left( \mathbf{T} \times \Spec(L) \right)$ depends on $\tau$ but only up to a sign. 
 
 \item \label{lem: kaz-section} The section  $\omega_\tau^{\otimes 2}$ of the bicanonical bundle is independent of $\tau$ and thus descends to a section of the bicanonical bundle on $\mathbf{T}$. 
 \end{enumerate}
\end{lemma}

\begin{definition}\label{Kazhdan form}
Let $\mathbf{T}$ be a $d$-dimensional torus defined over $F$.
We denote the section of the bicanonical bundle of $T$ introduced in \Cref{lem:kaz.form}\eqref{lem: kaz-section} by $\omega_{\mathbf{T}}^{K}$.    
\end{definition}   
}

\begin{definition}\label{def:kap}
For $x\in \fu^{rss}$, let $\bfU_x \subseteq \bfU$ be the stabilizer of $x$ under the adjoint action, let $\omega_{\bfU_x}^{K}$ be the section of the bicanonical bundle on $\bfU_x$ constructed in \Cref{Kazhdan form} and let $\mdef{\vert \omega_{\bfU_x}^{K} \vert^{\frac 12}}$ be the corresponding measure on $U_x$. Denote by $\mdef{(U_x)_0}$ the (unique) maximal compact subgroup of $U_x$ and define
\[
\mdef{\kappa_0(x)}:=\mdef{(\kappa_0)_V(x)}:=\int_{(U_x)_0} |\omega_{\bfU_x}|^{\frac 12}
\]
and 
\[
\mdef{\kappa(x)}:=\kappa_V(x):=\frac{\kappa_0(x)}{|\Delta(x)|^{\frac 12}}.
\]
\end{definition}

}

\section{The key ingredients and Proof of \Cref{thm:main}}\label{sec:key.ingr}
The proof of \Cref{thm:main} is based two ingredients that we  formulate below without proof. We only discuss them shortly. Then we deduce \Cref{thm:main}.  

The first ingredient is the following:

\begin{theorem}[\cite{Bes}]\label{thm:chv.kappa}
    Assume  that for any $f\in C_c^\infty(\fu)$ we have $$\frac{p_*(f\mu_\fu)}{\mu_{\fc}}\in L^{<\infty}(\fc).$$ 
    
    Then $\kappa\in L^{>1}_{loc}(\fu)$.
\end{theorem}

The next 3 sections will be dedicated to the proof of {the second ingredient}:

\begin{theorem}[\S\ref{sec:pf.kappa.chv}]\label{thm:kappa.chv}
Let $v\in V$ with $B(v,v)\neq 0$.  Let $W:=v^\bot$.
Assume that $\kappa_{W}\in L^{>1}_{loc}(\fu_W)$.
Then for any $f\in C_c^\infty(\fu)$ we have $$\frac{p_*(f\mu_\fu)}{\mu_{\fc}}\in L^{<\infty}(\fc).$$
\end{theorem}
\begin{proof}[Idea of the proof of \Cref{thm:kappa.chv}]
We choose an orthogonal $E$-basis on $V$ with $v$ being the last basis vector. Let $n=\dim_E V$. We fix the top-left $(n-1)\times (n-1)$ submatrix $A$ of a matrix in $\fu$ and consider $p$ as a map from the rest of the coefficients to $\fc$. Denote the resulting function by $p_A$. This is a quadratic map. We describe this map in \Cref{p and phi}.
Up to some minor modifications this is the map $\phi_A$ defined in \Cref{not:phiA}. 

We prove in \Cref{thm:Lr-bound} that for a generic $A$ the map $\phi_A$ pushes smooth compactly supported measure to a measure with  $L^{<\infty}$ density. Moreover, we bound (for any $r\in (1,\infty)$) the $L^{r}$ norm of this density in terms of $\kappa_W(A)$.

Finally we present $\frac{p_*(f\mu_\fu)}{\mu_{\fc}}$ as an integral over $A\in \fu_W$ of  the density functions $f_A$ of $(p_A)_*(\mu_A)$ for corresponding smooth compactly supported measures $\mu_A$. By the above, we can bound the $L^r$ norm of $f_A$. We 
use the integral triangle inequality (a.k.a Minkowski inequality) to deduce the result.
\end{proof}
The actual proof will be given in \S\ref{sec:pf.kappa.chv}.  We are now ready to deduce \Cref{thm:main}.

\begin{proof}[Proof of \Cref{thm:main}]
    The proof is by induction on $n=\dim_E(V)$. The base $n=1$ is obvious. So we assume that the statement is true for all quadratic spaces of dimension $n-1$ and prove it for the space $V$. Choose $v\in V$ s.t. $B(v,v)\neq 0$. By the induction hypothesis the statement is true for the quadratic space $W:=v^\bot$. Thus, by \Cref{thm:chv.kappa} the function $\kappa_W$ is in $L_{loc}^{\Dima{>}1}(\fu_W)$.
    By \Cref{thm:kappa.chv} the result follows.
\end{proof}

\section{Kottwitz heights}\label{sec:kottwitz}

For every quasi-affine $F$-variety $\mathbf{X}$ choose generators $f_1^{\mathbf{X}},\ldots,f_m^{\mathbf{X}}\in F[\mathbf{X}]$ and define a function $\hgt_\bfX:\bfX(\overline F) \rightarrow \mathbb{R}_{\geq 0}$ by
\[
\mdef{\hgt_\bfX(x)}=\max \left\{ 1,\log |f_1^\mathbf{X}(x)|,\ldots,\log |f_m^ \mathbf{X}(x)| \right\}.
\]

\begin{lemma} \label{Kottwitz lore}$ $ \begin{enumerate}
\item If $f\in F[\mathbf{X}]$ then there is a constant $C$ such that
\[
\log|f(x)| \leq C \cdot \hgt_\bfX(x),
\]
for every $x\in X$.
\item If $\phi: \mathbf{X} \rightarrow \mathbf{Y}$ is a morphism, then there is a constant $C$ such that 
\[
\hgt_\bfY(\phi(x)) \leq C \cdot \hgt_\bfX(x),
\]
for every $x\in X$.
\item If $g_1^ \mathbf{X},\ldots,g_k^ \mathbf{X}$ is another set of generators for $F[\mathbf{X}]$ and $\hgt'_\bfX$ is the Kottwitz height defined using $g_i^ \mathbf{X}$ then there is a constant $C$ such that 
\[
\frac1C \cdot \hgt_\bfX(x) \leq \hgt'_\bfX(x) \leq C \cdot \hgt_\bfX(x),
\]
for every $x\in X$.
\item If $\phi:\mathbf{X} \rightarrow \mathbf{Y}$ is finite then there is a constant $C$ such that
\[
\frac1C \cdot \hgt_\bfX(x) \leq \hgt_\bfY(\phi(x)) \leq C \cdot \hgt_\bfX(x)
,
\]
for every $x\in X$.
\end{enumerate} 
\end{lemma} 

\begin{proof} $ $ \begin{enumerate}
\item Since $f$ is a polynomial in $f_1^ \mathbf{X},\ldots,f_m^ \mathbf{X}$.
\item Follows from $(1)$.
\item Follows from $(2)$.
\item If $g\in F[\mathbf{X}]$ then there are $h_0,\ldots,h_{n-1}\in F[\mathbf{Y}]$ such that if $x\in X$ then
\[
g(x)^n+h_{n-1}(\phi(x))g(x)^{n-1}+\cdots+h_0(\phi(x))=0.
\]
Thus $|g(x)| \leq \max \left\{ 1,|h_0(\phi(x))|,\ldots,|h_{n-1}(\phi(x))| \right\}$ and, by $(1)$, the first inequality holds. The second inequality is $(2)$.
\end{enumerate} 
\end{proof} 

\begin{lemma}[cf. {\cite[Corollary 3.0.3]{AGKS_jets}}] \label{ht is integrable} Let $\mathbf{X} \subseteq \mathbb{A} ^n$ be Zariski open. Then the function $f:F^n \rightarrow \mathbb{R}_{\geq 0}$ defined by
\[
f(x):=\begin{cases} \hgt_\bfX(x) & x\in X \\ \infty & x\notin X  \end{cases}
\]
is in $L_{loc}^{<\infty}$.
\end{lemma} 

\section{A bound on the $L^r$-norm of a certain pushforward}\label{sec:Lr.bound}

\begin{notation}\label{not:phiA}
$ $
    For $A\in \fu$ define $\mdef{\phi_A}:=\mdef{\phi_{V,A}}:V\to F^n$ by $$\phi_A(v):=(B(A^iv,v)/\xi^i)_{i=0,\dots, n-1}$$   
\end{notation}

In this section we prove:
\begin{theorem}[The $L^r$-bound]
\label{thm:Lr-bound}
Let $1\le r<\infty$, let $\rho\in C^\infty_{c}(V)$. For \Dima{every} $A\in \fu^{rss}$ \Dima{we have}
$$
f_A:=\frac{(\phi_A)_*
\left(
\rho\,\mu_V
\right)
}{\mu_{F^n}}
\in L^r(F^n)
$$
Moreover, the function 
$$A\mapsto  \frac{\|f_{A}\|_{L^r(F^n)}}{\kappa(A)}$$ is in
$L^{<\infty}_{loc}(\fu)$
\end{theorem}
\begin{proof}[Idea of the proof] If $A\in \fu^{rss}$ and $c\in F^n$ \Eitan{is generic} then $\phi_A ^{-1}(c)$ is a torsor under the torus $\bfU_A$.
\Eitan{
This gives rise to 3 measures on the fiber:
\begin{enumerate}
    \item The Gelfand-Leray measure.
    \item The measure $|\omega_{\bfU_A}^{K}|^{\frac 12}$ described in  \Cref{def:kap}.
    \item The Haar measure normalized on the unique maximal compact subgroup. We call this measure ``the arithmetic measure". 
\end{enumerate}
The value $f_A(c)$ is the integral of $\rho$ w.r.t. the first measure. We show in \Cref{pullback of Gelfand Leray} below that the ratio of the first two measures is $|\Delta(A)|^{\frac 12}$. The ratio of the last two measures is $\kappa_0(A)$. Therefore, the ratio $\frac{f_A(c)}{\kappa(A)}$ is the integral of $\rho$ w.r.t. the third measure. 

We show that the  arithmetic measure of a set $B$ is polynomial in the log of the diameter of $B$ (see \Cref{arithmetic measure on tori}). One can deduce from this that the function $$(A,c)\mapsto \frac{f_A(c)}{\kappa(A)}$$ has only logarithmic growth at its singularities (see \eqref{eq:f/kap}).

This implies the assertion in view of \S\ref{sec:kottwitz}.
}
\end{proof}

\subsection{$F$-rational character of $\bfU_A$}

\Nir{\begin{definition} \label{characters of U_A} Let $A\in \fu^{rss}$. Under the isomorphisms $\mathbf{E}(F^{sep})\cong F^{sep} \times F^{sep}$, $\mathbf{V}(F^{sep})\cong (F^{sep})^n \times \left( (F^{sep})^n \right)^*$ from \S\ref{ssec:prel.uni}, $A$ corresponds to an element of the form $(X,-X^T)$, where $X\in \mathfrak{gl}_n^{rss}(F^{sep})$. \begin{enumerate}
\item If $e_1,\ldots,e_n$ is a basis of eigenvectors of $X$ and $e_1^*,\ldots,e_n^*$ is the dual basis, let $\mathbf{V}_i:=(\mathbf{E} \times \Spec F^{sep})(e_i,e_i^*) \subseteq \mathbf{V} \times \Spec F^{sep}$.
\item If $x\in \bfU_A$ then $x$ preserves each $\mathbf{V}_i$ and acts as multiplication by an element of the form $(\psi_i(x),\psi_i(x) ^{-1})$. The morphisms $\psi_i$ are $F^{sep}$-characters of $\bfU_A$  and generate the lattice $X^*(\bfU_A \times \Spec F^{sep})$.
\item $\Gal(F^{sep}/F)$ acts on the set $\left\{ \psi_1^{\pm1},\ldots,\psi_n^{\pm1}\right\}$. For every orbit $\mathcal{O}$, let 
\[
\varphi_\mathcal{O}:= \prod_{\psi \in \mathcal{O}} \psi.
\]
The lattice $X^*(\bfU_A)$ of $F$-characters is generated by the $\varphi_\mathcal{O}$s.
\item For $x\in U_A$ let
\[
\|x\|_A:=\max \left\{ | \varphi_\mathcal{O} (x) | \mid \text{$\mathcal{O}$ is an orbit} \right\}.
\]
\end{enumerate} 
\end{definition}

\begin{definition} Let $\bfT$ be a torus defined over $F$, let $T=\bfT(F)$, and let $T_0 \subseteq T$ be the maximal compact subgroup of $T$. The Haar measure on $T$ normalized so that the measure of $T_0$ is 1 is called the arithmetic measure of $T$ and is denoted by $\mu_T^{Ar}$.
\end{definition} 
}

\begin{lemma} \label{arithmetic measure on tori} 
\Eitan{For every
$A\in \fu^{rss}$ and $R\in\R_{>1}$}
\Eitan{we have:}
\[
\mu^{Ar}_{U_A}\left( \left\{ x\in U_A \mid \|x\|_A \leq R\right\} \right) \leq \Eitan {(1+2\log_\ell R)^{n}} 
\]
\end{lemma} 

\Nir{
\begin{proof} There are $m \leq n$ Galois orbits $\mathcal{O}_1,\ldots,\mathcal{O}_m \subseteq \left\{ \psi_1^{\pm1},\ldots,\psi_n^{\pm1} \right\}$ such that $\varphi_{\mathcal{O}_1},\ldots,\varphi_{\mathcal{O}_m}$ spans $X^*(\bfU_A) \otimes \mathbb{Q}$. Denote
\[
T_0:=\left\{ x\in U_A \mid (\forall \varphi \in X^*(\bfU_A))\, | \varphi(x)|=1 \right\}=\left\{ x\in U_A \mid (\forall i)\, | \varphi_{\mathcal{O}_i}(x)|=1 \right\}.
\]
Since $\varphi_{\mathcal{O}_1},\ldots,\varphi_{\mathcal{O}_m}$ spans $X^*(\bfU_A) \otimes \mathbb{Q}$, the map 
\[
xT_0\mapsto (\log_\ell | \varphi_{\mathcal{O}_1}(x)|,\ldots,\log_\ell | \varphi_{\mathcal{O}_m}(x)|)\in \mathbb{Z} ^m
\]
is injective on $U_A/T_0$. Thus, the set $\left\{ x\in U_A \mid \|x\|_A \leq R\right\}$ is contained in the union of at most $(1+2\log_\ell)^m$ cosets of $T_0$. 

Since $T_0$ is the maximal compact subgroup of $U_A$, $\mu^{Ar}_{U_A}(T_0)=1$ and the claim follows.
\end{proof} 
}
\begin{definition} $\mdef{\mathbf{Com^{rss}}}$ is the locus of pairs $(A,x)\in \uu^{rss} \times \bfU$ that commute.
\end{definition} 

\begin{lemma} \label{character bounds on tori} There is a natural number $C$ such that, for every $(A,x)\in Com^{rss}$,
\[
\log \| x\|_A \leq C \cdot \hgt_{\bf{Com^{rss}}}(A,x)
\]
\end{lemma} 
\Nir{
\begin{proof} Let $\Res_{E/F}\mathbb{P}_{\mathbf{E}}(\mathbf{V})$ be the Weil restriction of the projective space of $\mathbf{V}$ over $\mathbf{E}$. A point of $\Res_{E/F} \mathbb{P}_\mathbf{E}(\mathbf{V})(F^{sep})$ is in particular a free rank-1 sub $\mathbf{E}(F^{sep})$-module of $\mathbf{V}(F^{sep})$. Let $\mathbf{S} \subseteq \mathbf{Com^{rss}} \times \Res_{E/F}\mathbb{P}_\mathbf{E}(\mathbf{V})$ be the locus of triples $(A,x,\mathcal{L})$ such that $A$ preserves $\mathcal{L}$ and acts on $\mathcal{L}$ as multiplication by $\lambda \in \mathbf{E}$ satisfying $\lambda=-\overline{\lambda}$. The projection $\pi: \mathbf{S} \rightarrow \mathbf{Com^{rss}}$ is finite.

Let $g:\mathbf{S} \rightarrow \mathbf{E}^\times$ be the morphism such that if $(A,x,\mathcal{L})\in \mathbf{S}$ then $x$ acts on $\mathcal{L}$ as multiplication by $g(A,x,v)$. Let $h: \mathbf{E} \rightarrow \mathbb{A} ^1$ be the morphism $h(x)=x+\overline{x}$.

By \Cref{Kottwitz lore}, there is a  constant $C$ such that
\[
\pm \log| h(g(A,x,\mathcal{L})) | \leq C \hgt_{\bf{Com^{rss}}}(A,x).
\]
Let $(A,x)\in Com^{rss}$ and let $\psi_1,\ldots,\psi_n$ be as in \Cref{characters of U_A}. For each $i$ there is $\mathcal{L}_i\in \mathbb{P}_\mathbf{E}(\mathbf{V})(F^{sep})$ such that $g(A,x,\mathcal{L}_i)=(\psi_i(x),\psi_i(x) ^{-1})$. Thus,
\[
| \psi_i(x) | , | \psi_i(x) ^{-1} | \leq \max \left\{ | h(g(A,x,\mathcal{L}_i)|,1 \right\},
\]
and the claim follows.
\end{proof} 
}
\subsection{Height bounds}

\begin{lemma} \label{phi detrmines Gram} For every $i,j \geq 0$ there is a morphism $f_{i,j}:\uu \times \mathbb{A} ^n \rightarrow \mathbf{E}$ such that for every $(A,v)\in \uu \times \mathbf{V}$, $f_{i,j}(A,\phi_A(v))=B(A^iv,A^jv)$.
\end{lemma} 

\begin{proof} $f_{0,0},\ldots,f_{n-1,0}$ are powers of $\xi$ times coordinate projections. By Cayley--Hamilton, $f_{i,0}$ exists for $i \geq n$. Since $A^*=-A$, $f_{i,j}$ exists for  $i,j \geq 0$.
\end{proof} 

\begin{definition} $\mathbf{M} \subseteq \uu^{rss} \times \mathbf{V}$ is the locus of pairs $(A,v)$ such that $v,Av,\ldots,A^{n-1}v$ is a free $\mathbf{E}$-basis.
\end{definition} 

\begin{lemma} \label{constructing W} There is a constant $C$ and $\mathbf{W} \subseteq \uu^{rss} \times \mathbb{A} ^n$ such that if $(A,v)\in M$ then $(A,\phi_A(v))\in W$ and $\hgt_\bfM(A,v) \leq C \cdot \big(\hgt_\bfW(A,\phi_A(v))+\hgt_\bfV(v) \big)$.
\end{lemma} 

\begin{proof} Let $f_{i,j}$ be the morphisms from \Cref{phi detrmines Gram}. Let $\mathbf{W} \subseteq \uu^{rss} \times \mathbb{A} ^n$ be the locus of points $(A,c)$ such that the matrix $(f_{i,j}(A,c))$ is invertible. We show the claim holds for this $\mathbf{W}$.

Let $\mathbf{X} \subseteq \mathbf{M} \times \mathbb{A} ^n$ be the graph of the map $(A,v) \mapsto \phi_A(v)$. Since the Gram matrix of a basis in a non-degenerate Hermitian space is invertible, we get that (after a permutation of the coordinates) $\mathbf{X} \subseteq \mathbf{V} \times \mathbf{W}$. By \Cref{Kottwitz lore},
\Dima{
we can find constants $C_1,C_2,C$ s.t.
\[
\hgt_\bfM(A,v)\leq C_1\hgt_\bfX(A,v,\phi_A(v)) \leq C_2\hgt_{\bf{V \times W}}(A,v,\phi_A(v)) \leq C(\hgt_\bfV(v)+\hgt_\bfW(A,\phi_A(v))).
\]
}
\end{proof} 

\begin{definition} $ $ \begin{enumerate}
\item $\mathbf{R}_1 \subseteq \mathbf{M} \times \bfU$ is the variety of triples $(A,v,x)$ such that $x\in \bfU_A$.
\item $\mathbf{R}_2 \subseteq \mathbf{M} \times \mathbf{V}$ is the variety of triples $(A,v,u)$ such that $\phi_A(u)=\phi_A(v)$.
\item $a: \mathbf{R}_1 \rightarrow \mathbf{R}_2$ is the morphism $a(A,v,x)=(A,v,xv)$.
\end{enumerate} 
\end{definition}

\begin{lemma} \label{a is invertible} $ $ \begin{enumerate}
\item $a$ is an isomorphism. 
\item If $(A,v)\in M$ then the action map $a_{A,v}=a(A,v,-):\bfU_A \rightarrow \phi_A ^{-1} (\phi_A(v))$ is an isomorphism. 
\item There is a constant $C$ and $\mathbf{W} \subseteq \uu^{rss} \times \mathbb{A} ^n$ such that if $(A,v,u)\in R_2$ then $(A,\phi_A(v))\in W$ and 
\[
\hgt_\bfU(a_{A,v} ^{-1}(u)) \leq C(\hgt_\bfW(A,\phi_A(v))+\hgt_\bfV(v)+\hgt_\bfV(u)).
\]
\end{enumerate} 
\end{lemma} 

\begin{proof} $ $ \begin{enumerate}

\item  Let $(A,v,u)\in R_2$. Since $(A,v)\in M$, there are $\alpha_0,\ldots,\alpha_{n-1}\in E$ such that $u=\sum \alpha_i A^iv$. The matrix $x=\sum \alpha_i A^i$ commutes with $A$, maps $v$ to $u$, and hence maps $A^iv$ to $A^iu$ for $i \geq 0$. By \Cref{phi detrmines Gram}, the basis $(v,Av,\ldots,A^{n-1}v)$ is isometric to $(u,Au,\ldots,A^{n-1}u)$. Thus $x$ is an isometry, so $x\in \bfU_A$. By Cramer, the $\alpha_i$ (and hence $x$) are polynomial in $(A,v,u)$. This map is the inverse of $a$.
\item Follows from $(1)$.
\item We show the claim holds for the $\mathbf{W}$ constructed in \Cref{constructing W}.

By \Cref{constructing W} and \Cref{Kottwitz lore} there are constants $C_1,\ldots,C_4$ such that
\[
\hgt_\bfU(a_{A,v} ^{-1}(u)) \leq C_1\hgt_{\bf{R_1}}(a ^{-1} (A,v,u)) \leq C_2\hgt_{\bf{R_2}}(A,v,u) \leq 
\]
\[
\leq C_3(\hgt_{\bfM}(A,v)+\hgt_\bfV(u)) \leq C_4((A,\phi_A(v))+\hgt_\bfV(v)+\hgt_\bfV(u)).
\]
\end{enumerate} 
\end{proof} 

\subsection{The Gelfand--Leray form of $\phi_A$}

\begin{definition} Let $\mathbf{E}=\Res_{E/F} \mathbb{A}^1$. Fix an $F$-linear form $\eta \in \wedge ^2\Eitan{(E^*)}$ and let $\omega_{\mathbf{E}}\in \Omega ^2(\mathbf{E})$ be the translation-invariant 2-form such that $\omega_{\mathbf{E}}|_0= \eta$. Let $\omega_{\bfV}$ be the external wedge power of $\omega_{\mathbf{E}}$.
\end{definition} 

\begin{definition} Let $f: \mathbf{X} \rightarrow \mathbf{Y}$ be a morphism between smooth varieties, let $\omega_{\mathbf{X}}$ and $\omega_\mathbf{Y}$ be regular top differential forms on $\mathbf{X}$ and $\mathbf{Y}$, and let $x\in X$ be a point. Assume that $f$ is smooth at $x$. The short exact sequence
\[
0 \rightarrow T_x f ^{-1}(f(x)) \rightarrow T_x \mathbf{X} \overset{df}{\rightarrow} T_{f(x)}\mathbf{Y} \rightarrow 0
\]
induces an isomorphism 
\[
\RamiA{\iota} : \RamiA{\wedge^{top}((T_x \mathbf{X})^*)} \overset{\cong}{\rightarrow} \big( \RamiA{\wedge^{top} ((T_x f ^{-1}(f(x)))^*)} \big) \otimes \big( \RamiA{\wedge^{top} ((T_{f(x)} \mathbf{Y})^*)} \big),
\]
so there is a top form $\eta \in \RamiA{\wedge^{top} ((T_x f ^{-1}(f(x)))^*)}$ such that $$\RamiA{\iota} (\omega_\mathbf{X})= \eta \otimes \omega_\mathbf{Y}.$$ We denote this top form by $\frac{\omega_\mathbf{X}}{f^*\omega_\mathbf{Y}}\vert_x$ and call it the Gelfand--Leray form of $f$ at $x$.

This gives rise to a relative top-form on the smooth locus of $f$. We denote it by $\frac{\omega_\mathbf{X}}{f^*\omega_\mathbf{Y}}$.
\end{definition} 

\begin{lemma} \label{pullback of Gelfand Leray} There is a constant $c\in E$ such that if $(A,v)\in M$ then
\[
\left( a_{A,v}^* \left( \frac{\omega_{\bfV}}{\phi_A^*(dz_1 \wedge \cdots \wedge dz_n)}\right) \right) ^{\otimes 2}= c \cdot \Delta(A)\Eitan{^{-1}} \cdot \omega_{\bfU_A}^{\Dima{K}}.
\]
\end{lemma} 

\begin{proof} It suffices to show that there is a field extension $L/F$ and a trivialization $\tau: \bfU_A \times \Spec(L) \rightarrow \mathbb{G}_m^n$ such that
\begin{equation} \label{pull of Gelfand Leray}
a_{A,v}^* \left( \frac{\omega_{\bfV}}{\phi_A^*(dz_1\wedge \cdots \wedge dz_n)} \right)^{\Eitan{\otimes 2}} =c \cdot \Delta(A)^{-1} \cdot \tau^* \left( \frac{dt_1}{t_1}\wedge \cdots \wedge \frac{dt_n}{t_n} \right)^\Eitan{\otimes 2} .
\end{equation}

If $E$ is a quadratic field extension of $F$, let $L \supseteq E$ be a field containing all eigenvalues of $A$. If $E=F \times F$, write $A=(X,-X^T)$ as in \Cref{unitary if F x F} and let $L$ be a field containing all eigenvalues of $X$. In both cases, $\Dima{E\otimes_F L}=L \times L$, $\xi = (a,-a)$ for some $a\in L$, $V=L^n \times (L^n)^*$, $B=B_s$, and there is a diagonalizable matrix $Y$ such that $A=(Y,-Y^T)$.

If $e_1,\ldots,e_n$ is an eigenbasis of $Y$ with eigenvalues $\lambda_1,\ldots,\lambda_n$ then the vectors $(e_i,e_i^*)$ form an orthonormal basis of $V$. \Eitan{After multiplying $e_1$ by a scalar, we can further assume that the orthonormal basis $(e_i,e_i^*)$ is unimodular. Since unimodular matrices preserve top forms on $V$,} we may assume that $A$ is diagonal and thus $\phi_A$ has the form
\[
\phi_A \left( \sum \alpha_j v_j \right) = \left( \frac1{\xi ^i} B \left(  \sum \alpha_j A^i v_j,\sum \alpha_j v_j \right) \right) = \left( \sum_j \left( \frac{\lambda_j}{a} \right)^i \alpha_j \overline{\alpha_j}\right) \quad \alpha_j\in L \times L 
\]
i.e., $\phi_A$ is the composition of the map $\psi: L^n \times L^n \rightarrow L^n$
\[
\psi \big( (x_j),(y_j) \big) \mapsto \left( x_jy_j \right) 
\]
and the Vandermonde matrix $\left( \begin{matrix} \left( \frac{\lambda_j}{a} \right) ^i\end{matrix} \right)$. 
\Eitan{Denote the determinant of this Vandermonde matrix by $\cV(A)$. Note that $\cV(A)^2=\Delta(A)$. Now,}
\[
\phi_A^*(dz_1\wedge \cdots \wedge dz_n)=\det \left( \begin{matrix} \left( \frac{\lambda_j}{a} \right) ^i\end{matrix} \right) \cdot (x_1dy_1+y_1dx_1)\wedge \cdots \wedge (x_ndy_n+y_ndx_n)=
\]
\[
=a^{-\frac{n(n-1)}{2}} \cdot \Eitan{\cV(A)} \cdot (x_1dy_1+y_1dx_1)\wedge \cdots \wedge (x_ndy_n+y_ndx_n)
\]
and
\[
\frac{\omega_\bfV}{\phi_A^*(dz_1\wedge \cdots \wedge dz_n)}=a^{\frac{n(n-1)}2} \Eitan{\cV(A)^{-1}} \cdot \left( \frac{dx_1}{x_1} \wedge \cdots \wedge \frac{dx_n}{x_n} \right).
\]

On the other hand, the basis $v_j$ gives a trivialization of $\bfU_A$ under which the map $a_{A,v}$ becomes
\[
a_{A,v}(t_1,\ldots,t_n)=\big( (t_ia_i),(t_i ^{-1} b_i)\big)
\]
(for $v=((a_i),(b_i))$) and \eqref{pull of Gelfand Leray} follows.
\end{proof} 

\subsection{Proof of \Cref{thm:Lr-bound}}

We can assume that $\rho=1_\Omega$ where $\Omega \subseteq V$ is open compact.

Let $\mathbf{W}$ be as in \Cref{a is invertible}. Let $(A,c)\in W$. If $\phi_A ^{-1} (c)(F)\cap \Omega = \emptyset$ then $f_A(c)=0$. Otherwise, choose $v\in \phi_A ^{-1}(c)(F)\cap \Omega$. By \Cref{pullback of Gelfand Leray} there is a constant $C_1$ such that 
\[
f_A(c)=\frac{(\phi_A)_*(1_\Omega \mu_V)}{\lambda_{F^n}}(c)=\int_{\phi_A ^{-1} (c)(F)\cap \Omega} \left| \frac{\omega_\bfV}{\phi_A^*(dz_1\wedge \cdots \wedge dz_n)} \right|=
\]
\[
=\int_{a_{A,v} ^{-1}(\Omega)} \left| a_{A,v}^* \left( \frac{\omega_\bfV}{\phi_A^*(dz_1\wedge \cdots \wedge dz_n)} \right)\right|=C_1 \cdot | \Delta(A)|^{-\frac12} \cdot \Eitan{|\omega^{K}_{\bfU_x}|^{\frac 12}}(a_{A,v} ^{-1}(\Omega))=
\]
\[
=C_1 \cdot \kappa(A) \cdot \mu_{\Eitan{U_x}}^{Ar}\left( a_{A,v} ^{-1}(\Omega) \right).
\]
Let $x\in a_{A,v} ^{-1}(\Omega)$, i.e., $x=a_{A,v} ^{-1}(u)$ where $u\in \Omega$. By \Cref{character bounds on tori}, \Cref{Kottwitz lore}, and \Cref{a is invertible} there are constants $C_2,\ldots,C_5$ such that 
\[
\log \|x\|_A \leq C_2 \cdot \hgt_{\bf{Com^{rss}}}(A,x) \leq C_3 \cdot \big( \hgt_{\uu^{rss}}(A)+\hgt_\bfU(x) \big) \leq 
\]
\[
\leq C_4 \cdot \big( \hgt_{\uu^{rss}}(A)+\hgt_\bfV(v)+\hgt_\bfV(u)+\hgt_\bfW(A,c) \big) \leq C_5 \cdot \hgt_\bfW(A,c).
\]
By \Cref{arithmetic measure on tori}, there is a constant $C_6\in \Z_{>1}$ such that 
\begin{equation}\label{eq:f/kap}
f_A(c)\leq C_6 \cdot \kappa(A) \cdot \hgt_\bfW(A,c)^\Eitan{C_6}.    
\end{equation}
\Nir{By \Cref{ht is integrable}, $f_A$ is in $L^{<\infty}(F^n)$.}

Let $s,r\geq 1$ and let $\Omega' \subseteq \fu$ be compact open. 
\Eitan{We would like to show that the function 
$$A\mapsto  \frac{\| f_A \|_r}{\kappa(A)}$$ is in $L^s(\Omega')$. 
Without loss of generality we may assume that $s\geq r$.}

There is an open compact $\Omega'' \subseteq F^n$ such that $\phi_A(v)\in \Omega''$ for all $A\in \Omega'$ and $v\in \Omega$. 
\Nir{
We have:
\begin{align*}
\int_{A\in \Omega'} \left( \frac{\| f_A \|_r}{\kappa(A)} \right)^s \lambda_\fu
&\overset{\eqref{eq:f/kap}}{\leq} C_6^s \cdot \int_{A\in \Omega'} \| 1_{\Omega''}\fH_W(A,\cdot )^{C_6} \|_r^{s} \lambda_\fu \\
&\leq C_6^s \cdot \lambda_{F^n}(\Omega'')^{\frac sr-1} \cdot \int_{A\in \Omega'} \|1_{\Omega''}\fH_W(A,\cdot)^{C_6}\|_s^s \lambda_\fu \\
&= C_6^s \cdot \lambda_{F^n}(\Omega'')^{\frac sr-1} \int_{(A,c) \in \Omega' \times \Omega''} \fH_W(A ,c)^{s{C_6}} \lambda_{\fu \times F^n} \\
&\overset{\text{Lem \ref{ht is integrable}}}{<} \infty .
\end{align*}}

\section{Proof of \Cref{thm:kappa.chv}}\label{sec:pf.kappa.chv}

Let $V=W \oplus Ev$ and let $\beta=B(v,v) \in F.$ We assume $\beta \ne 0.$
Now for $w \in W$ define a linear functional $w^{*}$ on $W$ by 
$$w^{*}(y)=-\beta^{-1} \cdot B(y,w).$$
\Nir{The following is an easy computation:}
\begin{lemma}
    Every element $X \in \fu_V$ has the form 
$$X=X(A,w,d):=
\begin{pmatrix}
A & w \\
w^* & \xi d
\end{pmatrix},
$$
where $A\in\fu_{W}$, $w\in W$ and $d\in F$.
This gives an isomorphism
\begin{equation} \label{u is product} \fu_{V} \cong \fu_W \oplus W \oplus F.
\end{equation}
\end{lemma}

The next lemma provides a description of the Chevalley map 
 under the identification $\uc \cong \mathbb{A} ^n$ from \S\ref{subsec: chev.map}
in terms of the map 
$\phi_A:W \rightarrow F^{n-1}$ introduced in \Cref{not:phiA}. Recall that
for every $A\in \fu_W$, we defined a map $\phi_A:W \rightarrow F^{n-1}$ by
\[
\phi_A(w)=(B(A^iw,w)/ \xi^{i})_{i=0,\ldots,n-2}.
\]

\begin{lemma}\label{p and phi} Identify $\uc \cong \mathbb{A} ^n$ as in \S\ref{subsec: chev.map} using the invariants $(b_i)_{i=1}^n$. There are algebraic morphisms $f_1:\uu_W \times \mathbb{A} ^1 \rightarrow \mathbb{A} ^{n-1}$ and 
$f_{{2}}:\uu_W \times \mathbb{A} ^1 \rightarrow \GL_{n-1}$ such that

\[
p \left( \begin{pmatrix} A & w \\ w^* & \xi d \end{pmatrix} \right) =\begin{pmatrix} -\frac1{\xi}\trace(A)-d \\ f_1(A,d)+f_2(A,d) \left( \phi_A(w) \right) \end{pmatrix}
\]

Furthermore, $$det(f_2(A,d))=(\xi^{-2}\beta^{-1})^{n-1}.$$
\end{lemma} 

\begin{proof} \Nir{Let $X=X(A,w,d)\in \fu$. }Under the identification $\uc \cong \mathbb{A} ^n$, $p$ sends $X$ to the vector $\left( \frac{c_1(X)}{\xi},\frac{c_2(X)}{\xi ^2},\ldots,\frac{c_n(X)}{\xi ^n} \right)$, where $c_1(X),\ldots,c_n(X)$ are the coefficients of the characteristic polynomial of $X$,
\[
\det({\bf x}I-X)={\bf x}^n+\sum c_i(X){\bf x}^{n-i}.
\]

We have $c_1(X)=-\trace(X)=-\trace(A)- \xi d$. 

\Nir{
Schur's formula for the determinant of a block matrix implies that for $P \in M_{r}(F), r,s,q \in F$ we have:

\[
\det \left( \begin{pmatrix} P & q \\ r & s \end{pmatrix} \right) = s \det(P)-r \adj(P)q
\]
we obtain: 

\begin{multline*}  
\det({\bf x}I-X)=({\bf x}-\xi d)\det({\bf x}I-A)-(-w^{*})( \adj({\bf x}I-A)(-w))\\=
\det({\bf x}I-A)({\bf x}-\xi d)-w^{*}( \adj({\bf x}I-A)(w))\\=
\det({\bf x}I-A)({\bf x}-\xi d)+\beta^{-1}B(\adj({\bf x}I-A)w,w).
\end{multline*}

Thus,
\[
\det({\bf x}I-X)=\det({\bf x}I-A)({\bf x}-\xi d)+\beta^{-1}B(\adj({\bf x}I-A)w,w).\]
}
Let the adjugate matrix of ${\bf x}I-A$ be written as
\[
\adj({\bf x}I-A)=\sum_{k=0}^{n-2}C_k(A){\bf x}^k
\]
where $C_0(A),\ldots,C_{n-2}(A)\in \End(W)$. 

Applying the formula for $\det({\bf x}I-X)$ to \(X=X(A,w,d)\) and to
\(X=X(A,0,d)\), and subtracting the resulting identities, we obtain 
\[
\det({\bf x}I-X(A,w,d))-\det({\bf x}I-X(A,0,d))
=
\beta^{-1}B(\operatorname{adj}({\bf x}I-A)w,w).
\]

Comparing the coefficients of $x^{n-a}$ in \Nir{the last formula} we conclude that, for $a \geq 2$, $$c_a(X(A,w,d))-c_a(X(A,0,d))=\beta^{-1}B(C_{n-a}(A)w,w).$$

From \cite[Lemma 4.4]{AAGKS} we know that $C_{n-2}(A)=I$ and if $k<n-2$ then $C_k(A)-A^{n-k-2}$ is a linear combination of $C_{k+1}(A),\ldots,C_{n-2}(A)$. 
We define $f_1(A,d)$ to be the $n-1$ last coordinates of $p(X(A,0,d))$. 
Explicitly, 
$$
f_1(A,d)=
\left( \frac{c_2(X(A,0,d))}{\xi ^2},\ldots,\frac{c_n(X(A,0,d))}{\xi ^n} \right)
$$
We recall that for $A \in \fu_{W}$ 
\[
\det(xI_W-A)
=
x^{n-1}+c_1(A)x^{n-2}+\cdots+c_{n-1}(A),
\qquad
b_j(A):=\frac{c_j(A)}{\xi^j},
\]
and we define

\[
f_2(A,d)
=
\beta^{-1}\xi^{-2}
\begin{pmatrix}
1 & 0 & 0 & \cdots & 0\\
b_1(A) & 1 & 0 & \cdots & 0\\
b_2(A) & b_1(A) & 1 & \cdots & 0\\
\vdots & \vdots & \vdots & \ddots & \vdots\\
b_{n-2}(A) & b_{n-3}(A) & b_{n-4}(A)
& \cdots & 1
\end{pmatrix}.
\]

We checked above that the first coordinate of the Chevalley map is given by the formula stated in the lemma. The verification that the last $n-1$ coordinates of the Chevalley map $p$ satisfy the assertion, namely that
$$
\left( \frac{c_2(X)}{\xi ^2},\ldots,\frac{c_n(X)}{\xi ^n} \right)=
f_1(A,d)+f_2(A,d) \left( \phi_A(w) \right)
$$
is a direct calculation and is left to the reader.

\end{proof}

\begin{proof}[Proof of \Cref{thm:kappa.chv}]

We identify $\fc \cong F^n$ as in \S\ref{subsec: chev.map} and $\fu_{V} \cong \fu_W \oplus W \oplus F$ as in \eqref{u is product}. Since the conclusion of the theorem is unchanged upon multiplying any of
the Haar measures by a nonzero constant, we may and do normalize the measures as follows. Under the identifications
\[
\mathfrak u
\simeq
\mathfrak u_W\oplus W\oplus F
\]
from \eqref{u is product} and
\[
\mathfrak c\simeq F^n,
\]
we assume that
\[
\mu_{\mathfrak u}
=
\mu_{\mathfrak u_W}\otimes\mu_W\otimes\mu_F
\qquad\text{and}\qquad
\mu_{\mathfrak c}=\mu_{F^n}.
\]
We also write
\[
\mu_{W\oplus F}:=\mu_W\otimes\mu_F.
\]

It suffices to prove the theorem for $f=1_{X \times Y \times Z}$ where $X \subseteq \fu_W$, $Y \subseteq W$, and $Z \subseteq F$ are compact open subsets.

Fix $A\in \fu_W$ and let $g_A:W \oplus F \rightarrow F^n$ be the map $g_A(w,d)= p \left( \begin{pmatrix} A & w \\ w^* & \xi d \end{pmatrix} \right)$. By \Cref{p and phi},
\[
g_A(w,d)=\begin{pmatrix} -\frac1{\xi}\trace(A)-d \\ f_1(A,d)+f_2(A,d) \left( \phi_A(w) \right) \end{pmatrix}.
\]

Let $A \in \fu_{W}^{rss}$,
the push forward $(g_A)_*(1_{Y \times Z} \cdot \mu_{W \oplus F})$ is absolutely continuous with respect to $\mu_{F^n}$ \Nir{(e.g. by \Cref{thm:Lr-bound})}. 
We first obtain a formula for the $L^r$ norm of the density function for this measure.
We will show that this norm of the density function
\[
h_A:=
\frac{(g_A)_*(1_{Y\times Z}\mu_{W\oplus F})}
     {\mu_{F^n}}
\]
is expressible in terms of the $L^r$ norm of the density function studied in \Cref{thm:Lr-bound}. Namely, let
\[
q_A:=
\frac{(\phi_A)_*(1_Y\mu_W)}
     {\mu_{F^{n-1}}}.
\]
Then for any $A \in \fu_{W}^{rss}$ we \Nir{claim}:
\begin{equation}\label{eq:h_A}
\|h_A\|_{r}
=
C_r\|q_A\|_{r},
\end{equation}
where
\[
C_r
=
\mu_F(Z)^{1/r}
\left|\xi^{-2}\beta^{-1}\right|^{(n-1)(\frac{1}{r}-1)}
\]
is independent of \(A\).

To show that, 
let $d_A(a)=-a-\frac{1}{\xi}\operatorname{tr}(A)$ and let $M_A(a)=f_2(A,d_A(a)).$ 
If $a\in F$ and $b\in F^{n-1}$, the density of $(g_A)_*(1_{Y \times Z} \cdot \mu_{W \oplus F})$ at $\begin{pmatrix} a \\ b \end{pmatrix} \in F^n$ is

\begin{align*}
h_A\left(
\begin{pmatrix}
a\\
b
\end{pmatrix}
\right)=\frac{(g_A)_*(1_{Y \times Z}\mu_{W \oplus F})}{\mu_{F^n}}
\left(
\begin{pmatrix}
a\\
b
\end{pmatrix}
\right)
&=
|\det(M_A(a))|^{-1}
1_Z(d_A(a))
\nonumber\\
&\cdot
q_A(
M_A(a))^{-1}
(b-f_1(A,d_A(a)))).
\end{align*}
Note that
\[
D:=\left|\det M_A(a)\right|
=\left|\det f_2(A,d_A(a))\right|
  =\left|\xi^{-2}\beta^{-1}\right|^{n-1}.
\]
 is independent of \(A\) and \(a\).

For each fixed \(a\in F\), translation by
\(f_1(A,d_A(a))\) and the change of
variables
\[
b=M_A(a)u
\]
give
\[
\begin{aligned}
\int_{F^{n-1}}|h_A(a,b)|^r\,db
={}&
1_Z\left(-a-\frac{\operatorname{tr}(A)}{\xi}\right)
D^{1-r}
\|q_A\|_{L^r(F^{n-1})}^r.
\end{aligned}
\]
Since the map
\[
a\longmapsto d_A(a)=-a-\frac{\operatorname{tr}(A)}{\xi}
\]
preserves Haar measure on \(F\), it follows that
\[
\|h_A\|_{r}
=
\mu_F(Z)^{1/r}
D^{1/r-1}
\|q_A\|_{L^r}=C_r\|q_A\|_{r},
\]
as claimed. 

We can now show the statement of \Cref{thm:kappa.chv}.
We can assume without loss of generality that $1 \leq r< \infty$.
By \Cref{thm:Lr-bound} applied to the Hermitian space $W$ and the function $\rho=1_{Y} \in C^{\infty}_c(W)$ we obtain that the function
\[
\Phi_r(A):=
\frac{1}{\kappa_W(A)}
\left\|
\frac{(\phi_A)_*(1_Y\mu_W)}
     {\mu_{F^{n-1}}}
\right\|_{r}
\]
belongs to \(L^{<\infty}_{\mathrm{loc}}(\mathfrak u_W)\). \Nir{We have}
\begin{equation}\label{eq:q_A}
\left\|
q_A
\right\|_{r}=
\left\|
\frac{(\phi_A)_*(1_Y\mu_W)}
     {\mu_{F^{n-1}}}
\right\|_{r}
=
\kappa_W(A)\Phi_r(A).
\end{equation}

Now by \href{https://grokipedia.com/page/Tonelli's_theorem}{Tonelli's theorem}, 
\[
p_*(f\mu_{\mathfrak u})
=
\int_{A\in X}
(g_A)_*(1_{Y\times Z}\mu_{W\oplus F})
\,d\mu_{\mathfrak u_W}(A).
\]
\Eitan{
Minkowski's inequality for $L^r(F^n)$ 
implies that:
\[
\left\|
\frac{p_*(f\mu_{\mathfrak u})}{\mu_{\mathfrak c}}
\right\|_{r}
\leq
\int_X ||h_A||_{r}
d\mu_{\mathfrak u_W}(A)
\overset{\eqref{eq:q_A}\eqref{eq:h_A}}{=}
C_r\int_X \kappa_W(A)\Phi_r(A)\,
d\mu_{\mathfrak u_W}(A),
\]
}
Since \(\kappa_W\in L^{>1}_{\mathrm{loc}}(\mathfrak u_W)\)
and \(X\) is compact, there exists \(s>1\) such that
\(\kappa_W\in L^s(X)\). Let \(s'=s/(s-1)\). Since
\(\Phi_r\in L^{<\infty}_{\mathrm{loc}}(\mathfrak u_W)\), we have
\(\Phi_r\in L^{s'}(X)\). Therefore, by H\"older's inequality,
\[
\begin{aligned}
\left\|
\frac{p_*(f\mu_{\mathfrak u})}{\mu_{\mathfrak c}}
\right\|_{r}
\leq
C_r\int_X \kappa_W(A)\Phi_r(A)\,
d\mu_{\mathfrak u_W}(A)
&\leq
C_r
\left(\int_X \kappa_W(A)^s\,
d\mu_{\mathfrak u_W}(A)\right)^{1/s}\\
&\qquad\cdot
\left(\int_X \Phi_r(A)^{s'}\,
d\mu_{\mathfrak u_W}(A)\right)^{1/s'}
<\infty.
\end{aligned}
\]
Since \(1 \leq r<\infty\) was arbitrary, the result follows.

\end{proof}

\section{The group case -- proof of \Cref{thm:main.grp}}\label{sec:grp.case}
\begin{proof}[Idea of the proof]
We use the Cayley transform to cover $\bfU$ by open sets which are isomorphic to open sets in $\uu$ in a way that is compatible with the maps $\cP$ and $p$ (see \Cref{thm:Cayley} below). This allows us to use the existence of a partition of unity and deduce  the result from \Cref{thm:main}.
\end{proof}
\subsection{The categorical quotient}\label{ssec:cat.qu}
Let us describe explicitly the categorical quotient $\bfU//\bfU$.
We set $n:=\dim(V)$.

\begin{notation}\label{not:cP}
    $ $
    \begin{itemize}
        \item Let $\mdef{\mathbf{P}}:=  \mdef{\mathbf{P}_n}$  be the variety of monic polynomials \Dima{of degree $n$}, with coefficients in $\bfE$, considered as a variety defined over $F$.

        \item Let $\mdef{\bfP^{\times}}:=\{f\in \bfP\, \vert \, f(0)\Nir{\text{ is invertible}}\},$ and denote by $\mdef{\theta}:=\mdef{\theta_n}:\mathbf{P^{\times}}\to \mathbf{P^\times}
        $ 
        the involution defined by $$(\theta(f))(x):=x^nf(x^{-1})/f(0).$$
        \item Denote $\mdef{\bfC}:=\mdef{\bfC_n}:=\Dima{\{f\in \bfP^\times\,\vert \, \theta(f)=\overline{f}\}}$.
        \item Let $\mdef{\bfG}=\mdef{\bfG_\bfV}=\Aut_E(\bfV)$ considered as an algebraic group over $F$.
        \item Let $\mdef{\pi_\bfV}:=\mdef{\pi}: \bfG\to \Dima{\bfP^{\times}}$ be the map sending a matrix to its characteristic polynomial.
    \end{itemize}
\end{notation}
The following lemma is standard:
\begin{lemma}\label{lem:cP1}
$\pi(\bfU)=\bfC$.
\end{lemma}

\begin{lemma}\label{lem:cP2}
Consider the restriction of $\pi$ to $\bfU$ as a map to $\bfC$. Then this map is the categorical quotient map $\bfU\to \bfU//\bfU$ with respect to the adjoint action. 
\end{lemma}
\begin{proof}
    It is analogous to \Cref{lem:qust.lie}. 
\end{proof}
\begin{notation}\label{not:cP2}
    We will identify $\bfU//\bfU$ with $\bfC$ and denote the the categorical quotient map by $\mdef{\cP_V}:=\mdef{\cP}:\bfU\to \bfC$.
\end{notation}
\subsection{The local  structure of $\cP$}

The main fact that allows us to deduce \Cref{thm:main.grp} from \Cref{thm:main} is the following:

\begin{theorem}\label{thm:Cayley}
There exist:
\begin{itemize}
    \item a Zariski open cover \(\mathbf U=\bigcup_{i=1}^N\mathbf X_i\);
    \item an open subset \(\mathbf X\subset \mathfrak u\);
    \item an open subset \(\mathbf T\subset \uc\);
    \item open subsets \(\mathbf C_i\subset \bfC\);
    \item isomorphisms \(\varphi_i:\mathbf X_i\xrightarrow{\sim}\mathbf X\);
    \item isomorphisms
    $
    j_i:\bfT\xrightarrow{\sim}\bfC_i
    $
\end{itemize}
such that  for every \(i\), 
we have 
the following commutative diagram:
\[
\begin{tikzcd}
& \mathbf X_i \arrow[r, "\varphi_i"', "\sim"] \arrow[d]
\arrow[dl, "\cP|_{\mathbf X_i}"'] 
&
\mathbf X \arrow[dr, "p|_{\mathbf X}"]\arrow[d]
&
\\
\bfC
&
\mathbf C_i \arrow[l, hook']
&
\mathbf T \arrow[l, "j_i", "\sim"']\arrow[r, hook]
&
\uc
\end{tikzcd}
\]\end{theorem}
This theorem is rather standard, for completeness we include its proof here. We start with some preparations:

\begin{notation}\label{not:Cayley-polynomial}
$ $
\begin{enumerate}
    \item We identify \(\uc\) with the following subvariety of \(\bfP\):
    \[
    \uc =
    \left\{
    f(x)=x^n+\sum_{i=1}^n a_i x^{n-i}\in \bfP
    \;\middle|\;
    \overline{a_i}=(-1)^i a_i \text{ for all } i
    \right\}.
    \]    
    Under this identification,
    \[
    p(A)=\chi_A .
    \]

    \item Let
    \[
    \mdef{\bfT}:=\{f\in \uc\mid f(1)\Dima{\text{ is invertible}}\}.
    \] 
    \item Let
    \[
    \mdef{\bfX}:=p^{-1}(\bfT)
    =
    \{A\in\uu\mid \det(I-A)\Dima{\text{ is invertible}}\}.
    \]
    \item Let
    \[
    \mdef{\bfX_0}:=
    \{g\in\bfU\mid \det(g+I)\Dima{\text{ is invertible}}\}.
    \]

    \item Define the Cayley transform
    \[
    \mdef{\operatorname{cay}}:\bfX\longrightarrow \bfX_0,
    \qquad
    A\longmapsto (I+A)(I-A)^{-1}.
    \]
    \item Let
    \[
    \mdef{\bfC_0}:=
    \{q\in\bfC\mid q(-1)\Dima{\text{ is invertible}}\}.
    \]

    \item Define
    \[
    \mdef{\operatorname{cay}}_\fc:\bfT\longrightarrow \bfC_0
    \]
    by the following formula. For \(f\in\bfT\), set
    \[
    {\operatorname{cay}_\fc}(f)(x):=
    \frac{(x+1)^n}{f(1)}
    f\left(\frac{x-1}{x+1}\right).
    \]
    This is a monic polynomial of degree \(n\), hence an element of \(\bfP\). It is easy to see that it is also an element in $\bfC_0$.
\end{enumerate}
\end{notation}

\begin{lemma}\label{lem:Cayley-j-properties}
With the notation above, the following hold.
\begin{enumerate}
    \item\label{it:Cayley-cay-isom}
 The map
    \[
    \operatorname{cay}:\bfX\longrightarrow \bfX_0
    \]
    is an isomorphism. Its inverse is
    \[
    g\longmapsto (g-\Dima{I})(g+\Dima{I})^{-1}.
    \]

    \item\label{it:Cayley-cay-fc-isom}
 The map
    \[
    \operatorname{cay}_\fc:\bfT\longrightarrow \bfC_0
    \]
    is an isomorphism. Its inverse is given as follows: for \(q\in\bfC_0\),
    \[
    \operatorname{cay}_\fc^{-1}(q)(t)
    =
    \frac{(-1)^n(1-t)^n}{q(-1)}
    q\left(\frac{1+t}{1-t}\right).
    \]

    \item\label{it:Cayley-compatibility}
 The Cayley transform is compatible with the two quotient maps:
    \[
    \cP\circ \operatorname{cay}
    =
    \operatorname{cay}_\fc\circ p|_{\bfX}.
    \]
\end{enumerate}
\end{lemma}
\begin{proof}
We prove the three assertions separately.

\begin{enumerate}
    \item Write \(A^*\) for the adjoint of \(A\) with respect to \(B\).
    Thus \(A\in\uu\) means \(A^*=-A\). If $A\in \bfX(F)$ and 
    \[
    g:=\operatorname{cay}(A)=(\Dima{I}+A)(\Dima{I}-A)^{-1},
    \]
    then
    \[
    g^*
    =
    \bigl((\Dima{I}-A)^{-1}\bigr)^*(\Dima{I}+A)^*
    =
    (\Dima{I}+A)^{-1}(\Dima{I}-A).
    \]
    Hence
    \[
    g^*g
    =
    (\Dima{I}+A)^{-1}(\Dima{I}-A)(\Dima{I}+A)(\Dima{I}-A)^{-1}
    =
    \Dima{I},
    \]
    so \(g\in\bfU\). Moreover,
    \[
    g+\Dima{I}
    =
    (\Dima{I}+A)(\Dima{I}-A)^{-1}+\Dima{I}
    =
    2(\Dima{I}-A)^{-1},
    \]
    which is invertible since \(\operatorname{char}F\neq 2\). Thus
    \(g\in\bfX_0\).

    Conversely, let \(g\in\bfX_0\), and set
    \[
    A:=(g-\Dima{I})(g+\Dima{I})^{-1}.
    \]
    Since \(g\in\bfU\), we have \(g^*=g^{-1}\). Therefore
    \[
    A^*
    =
    (g^{-1}+\Dima{I})^{-1}(g^{-1}-\Dima{I})
    =
    (g+\Dima{I})^{-1}(\Dima{I}-g)
    =
    -A.
    \]
    Hence \(A\in\uu\). Also,
    \[
    \Dima{I}-A
    =
    \Dima{I}-(g-\Dima{I})(g+\Dima{I})^{-1}
    =
    2(g+\Dima{I})^{-1},
    \]
    so \(A\in\bfX\). Direct substitution gives
    \[
    \operatorname{cay}\bigl((g-\Dima{I})(g+\Dima{I})^{-1}\bigr)=g
    \]
    and
    \[
    \bigl(\operatorname{cay}(A)-\Dima{I}\bigr)
    \bigl(\operatorname{cay}(A)+\Dima{I}\bigr)^{-1}=A.
    \]
    Thus \(\operatorname{cay}:\bfX\to\bfX_0\) is an isomorphism with inverse
    \[
    g\longmapsto (g-\Dima{I})(g+\Dima{I})^{-1}.
    \]

    \item We first check that \(\operatorname{cay}_\fc\) maps \(\bfT\) to
    \(\bfC_0\). Let \(f\in\bfT\). Since \(f\in\uc\), we have
    \[
    \overline{f(t)}=(-1)^n f(-t).
    \]
    Put
    \[
    q(x):=\operatorname{cay}_\fc(f)(x)
    =
    \frac{(x+1)^n}{f(1)}
    f\left(\frac{x-1}{x+1}\right).
    \]
    Then \(q\) is monic of degree \(n\), and
    \[
    q(-1)=\frac{(-2)^n}{f(1)}
    \]
    \Dima{is invertible.} Also
    \[
    q(0)=\frac{f(-1)}{f(1)}
    \]
    \Dima{is invertible,} because \(f(1)\) \Dima{is invertible,} and
    \[
    \overline{f(-1)}=(-1)^n f(1).
    \]
    Now
    \[
    \begin{aligned}
    \overline{q(x)}
    &=
    \frac{(x+1)^n}{\overline{f(1)}}
    \overline{
    f\left(\frac{x-1}{x+1}\right)
    }                                                    \\
    &=
    \frac{(x+1)^n}{f(-1)}
    f\left(\frac{1-x}{1+x}\right).
    \end{aligned}
    \]
    On the other hand,
    \[
    \begin{aligned}
    \frac{x^nq(x^{-1})}{q(0)}
    &=
    \frac{x^n}{q(0)}
    \frac{(x^{-1}+1)^n}{f(1)}
    f\left(\frac{x^{-1}-1}{x^{-1}+1}\right)        \\
    &=
    \frac{(x+1)^n}{f(-1)}
    f\left(\frac{1-x}{1+x}\right).
    \end{aligned}
    \]
    Hence
    \[
    \overline q=\theta(q),
    \]
    so \(q\in\bfC\). Since \(q(-1)\) \Dima{is invertible}, we get \(q\in\bfC_0\).

    Now define, for \(q\in\bfC_0\),
    \[
    I(q)(t):=
    \frac{(-1)^n(1-t)^n}{q(-1)}
    q\left(\frac{1+t}{1-t}\right).
    \]
    
    The right-hand side is a polynomial of degree \(n\). Its leading
    coefficient is \(1\), because the coefficient of \(t^n\) in
    \[
    (1-t)^n q\left(\frac{1+t}{1-t}\right)
    \]
    is \((-1)^nq(-1)\). Also,
    \[
    I(q)(1)=\frac{(-1)^n2^n}{q(-1)} 
    \]
    \Dima{is invertible.}
    Since \(q\in\bfC\), we have \(\overline q=\theta(q)\), i.e.
    \[
    \overline{q(x)}=\frac{x^nq(x^{-1})}{q(0)}.
    \]
    Using this identity, one obtains
    \[
    \overline{I(q)(t)}
    =
    (-1)^n I(q)(-t).
    \]
    Hence \(I(q)\in\uc\), and since \(I(q)(1)\) \Dima{is invertible}, we have
    \(I(q)\in\bfT\).

    Finally, $I$ is inverse to $\operatorname{cay}_{\fc}$. This follows by direct substitution from the fact that the two
    fractional linear transformations
    \[
    t\longmapsto \frac{1+t}{1-t},
    \qquad
    x\longmapsto \frac{x-1}{x+1}
    \]
    are inverse. 

    \item Let \(A\in\bfX\). Since \(p(A)=\chi_A\), it is enough to compute
    the characteristic polynomial of \(\operatorname{cay}(A)\). We have
    \[
    \begin{aligned}
    \chi_{\operatorname{cay}(A)}(x)
    &=
    \det\left(x\Dima{I}-\operatorname{cay}(A)\right)                  \\
    &=
    \det\left(x\Dima{I}-(\Dima{I}+A)(\Dima{I}-A)^{-1}\right)                         \\
    &=
    \frac{
    \det\bigl(x(\Dima{I}-A)-(\Dima{I}+A)\bigr)
    }{
    \det(\Dima{I}-A)
    }                                                         \\
    &=
    \frac{
    \det\bigl((x-1)\Dima{I}-(x+1)A\bigr)
    }{
    \chi_A(1)
    }                                                         \\
    &=
    \frac{(x+1)^n}{\chi_A(1)}
    \chi_A\left(\frac{x-1}{x+1}\right).
    \end{aligned}
    \]
    By the definition of \(\operatorname{cay}_\fc\), this is precisely
    \[
    \operatorname{cay}_\fc(\chi_A)(x).
    \]
    Since \(p(A)=\chi_A\) and \(\cP(g)=\chi_g\), we get
    \[
    \cP(\operatorname{cay}(A))
    =
    \operatorname{cay}_\fc(p(A)), \text{ thus }\cP\circ \operatorname{cay}
    =
    \operatorname{cay}_\fc\circ p|_{\bfX}.
    \]
\end{enumerate}
\end{proof}

\Dima{
\begin{lemma}\label{lem:BezInv}
    Let $f\in C$. Let $x_1,\dots,x_{n+1}\in E^1$ be pairwise distinct elements. Then for some $1\leq i\leq n+1$, $f(x_i)$ is invertible.
\end{lemma}
\begin{proof}
    If $E$ splits then this is straightforward. This implies the general case by extension of scalars.
\end{proof}

}

\begin{proof}[Proof of \Cref{thm:Cayley}]
$ $
\begin{enumerate}[Step 1:]
    \item Construction of the data\\
Let
\[
E^1:=\{\lambda\in E^\times\mid \lambda\overline{\lambda}=1\}.
\]
This set is clearly infinite. 
Choose pairwise distinct
\[
\lambda_1,\ldots,\lambda_{n+1}\in E^1 .
\]

For \(\lambda\in E^1\), let
\[
m_\lambda:\bfU\longrightarrow \bfU,
\qquad
g\longmapsto \lambda g .
\]
This is multiplication by a central element of \(\bfU\). It induces
\[
\overline m_\lambda:\bfC\longrightarrow \bfC,
\qquad
q(x)\longmapsto \lambda^n q(\lambda^{-1}x).
\]
Set
\[
\bfX_\lambda:=m_\lambda(\bfX_0)
=
\{g\in\bfU\mid \det(g+\lambda\Dima{I})\text{ \Dima{is invertible}}\}
\]
 and
\[
\bfC_\lambda:=\overline m_\lambda(\bfC_0).
\]

We define
\[
N:=n+1,
\qquad
\bfX_i:=\bfX_{\lambda_i}, \qquad \bfC_i:= \bfC_{\lambda_i},
\]

\[
\varphi_i:\bfX_i\xrightarrow{\sim}\bfX,
\qquad
\varphi_i(g):=\operatorname{cay}^{-1}(\lambda_i^{-1}g), 
\]

and
\[
j_i
:=
\overline m_{\lambda_i}\circ \operatorname{cay}_\fc|_{\bfT} .
\]

\item Verification of the conditions.\\
The maps $\phi_i$ are  isomorphisms by Lemma~\ref{lem:Cayley-j-properties}
(\ref{it:Cayley-cay-isom}), because \(m_{\lambda_i}^{-1}\) identifies
\(\bfX_i\) with \(\bfX_0\).
Similarly $j_i$ are isomorphisms by Lemma~\ref{lem:Cayley-j-properties}(\ref{it:Cayley-cay-fc-isom}) 

\Dima{By \Cref{lem:BezInv} \(\bfX_i\) cover \(\bfU\).}
\Dima{Let} \(g\in\bfX_i\), and put
\[
A:=\varphi_i(g).
\]
Then
\[
g=\lambda_i\operatorname{cay}(A).
\]
Therefore,
\[
\begin{aligned}
\cP(g)
&=
\cP\bigl(m_{\lambda_i}(\operatorname{cay}(A))\bigr)  
=
\overline m_{\lambda_i}\bigl(\cP(\operatorname{cay}(A))\bigr) \\
&=
\overline m_{\lambda_i}
\bigl(\operatorname{cay}_\fc(p(A))\bigr)
\qquad
\text{by Lemma~\ref{lem:Cayley-j-properties}
(\ref{it:Cayley-compatibility})} \\
&=
j_i
\bigl(p(\varphi_i(g))\bigr).
\end{aligned}
\]
This implies the required  commutativity of the diagram.
\end{enumerate}

\end{proof}

\subsection{Proof of \Cref{thm:main.grp}}\label{Subsec-proof-of-thmB}
To streamline the proof of \Cref{thm:main.grp} we introduce some terminology.

\begin{definition}\label{def:FRS}
$ $
\begin{itemize}
    \item {For an open set $\Omega\subset F^d$, a  \tdef{smooth measure} on $\Omega$ is a measure that can be written as $f\mu$ where $\mu$ is the Haar measure on $F^d$ and $f\in C^\infty(\Omega)$.} For an $F$-analytic manifold $X$, a smooth measure on $X$ is a measure that 
{is smooth on every local chart.}
\item For an $F$-analytic manifold $X$ we denote by $\mdef{\cM^\infty(X)}$ the space of smooth measures on $X$ and by  $\mdef{\cM_c^\infty(X)}$  the space of smooth compactly supported measures on $X$.  
\item By an $\mdef{L_{{loc}}^q}$ function 
on an $F$-analytic manifold we mean function which is $L_{{loc}}^q$ with respect to ({every}) smooth measure with full support. 
\item These \Eitan{allow} us to use the notations $\mdef{L_{loc}^{<q}}, \mdef{L_{loc}^{>q}}$ in this generality.

\item
    Let $\phi:\bfX\to \bfY$ be a  morphism of smooth algebraic varieties. Let $X:=\bfX(F)$ and $Y:=\bfY(F)$.
        We say that $\phi$ is \tdef[almost analytically FRS, almost an-FRS]{almost analytically FRS} (in short almost an-FRS) if 
        $$\phi_*(\cM_c^\infty(X)) \subset L^{<\infty}_{loc}(Y) \cM^\infty(Y).$$
\end{itemize}
\end{definition}

\Nir{The following lemma follows from a partition of unity argument:} 

\begin{lemma}\label{lem:almost-anFRS-source-local}
Let $\phi:\bfX\to\bfY$ be a morphism of smooth algebraic varieties, and let
$$
\bfX=\bigcup_{i=1}^N \bfX_i
$$
be a finite Zariski open cover. Then 
TFAE:
\begin{itemize}
    \item  $\phi$ is almost an-FRS
    \item  For each $i$ the map $\phi|_{\bfX_i}$
 is almost an-FRS
\end{itemize}

\end{lemma}

\Nir{The next lemma is straightforward:}

\begin{lemma}\label{lem:almost-anFRS-open-target-restriction}
Let $\phi:\bfX\to\bfY$ be a morphism of smooth algebraic
varieties and let $\psi:\bfY\to \bfY_1$  be an open embedding. Then TFAE:

\begin{itemize}
    \item $\phi$ is almost an-FRS. 
    \item $\psi\circ \phi$ is almost an-FRS. 
\end{itemize}
\end{lemma}

\begin{proof}[Proof of \Cref{thm:main.grp}]
We need to show that $\cP$ is almost an-FRS. By \Cref{thm:Cayley}, there are finitely many open subsets
\(\bfX_i\subset \bfU\), an open subset \(\bfX\subset \uu\), isomorphisms
\[
\varphi_i:\bfX_i\xrightarrow{\sim}\bfX,
\]
and isomorphisms
\[
j_i:\bfT\xrightarrow{\sim}\bfC_i\subset \bfC
\]
such that  for every \(i\), we have the following commutative diagram:
\[
\begin{tikzcd}
& \mathbf X_i \arrow[r, "\varphi_i"', "\sim"] \arrow[d, "\cP_{\mathbf X_i}"]
\arrow[dl, "\cP|_{\mathbf X_i}"'] 
&
\mathbf X \arrow[dr, "p|_{\mathbf X}"]\arrow[d,"p_\bfX"]
&
\\
\bfC
&
\mathbf C_i \arrow[l, hook']
&
\mathbf T \arrow[l, "j_i", "\sim"']\arrow[r, hook]
&
\uc
\end{tikzcd}
\]
By \Cref{lem:almost-anFRS-source-local} it is enough to show that 
$\cP|_{\bfX_i}$  is almost an-FRS.

For this, by \Cref{lem:almost-anFRS-open-target-restriction}, it is enough to show that 
$p_{\bfX_i}$  is almost an-FRS. Thus it is  enough to show that $j_i^{-1}\circ \cP_{\bfX_i}=p_{\bfX} \circ \varphi_i$  is almost an-FRS.  For this, it is enough to show that 
$p_{\bfX}$ is almost an-FRS. By \Cref{lem:almost-anFRS-open-target-restriction} it is enough to show that  $p|_{\bfX}$ is almost an-FRS. This follows from \Cref{thm:main} using \Cref{lem:almost-anFRS-source-local}.
\end{proof}

\begingroup
  \let\clearpage\relax
  \let\cleardoublepage\relax 
  \printindex
\endgroup

\bibliographystyle{alphaurl}
\bibliography{Ramibib}

\end{document}